\documentclass[11pt]{article}
\usepackage{arxiv}
\usepackage[utf8]{inputenc}
\usepackage{graphicx}
\usepackage{epstopdf}% To incorporate .eps illustrations using PDFLaTeX, etc.
\usepackage[caption=false]{subfig}% Support for small, `sub' figures and tables

\usepackage[sort&compress]{natbib}
 \bibpunct[, ]{[}{]}{,}{n}{}{,}%
\usepackage{amsmath}
\usepackage{amssymb}
\usepackage{amsthm}
\usepackage{algorithm}
\usepackage{algpseudocode}
\usepackage[hyphens]{url}
\usepackage{hyperref}
\usepackage{booktabs}

\newtheorem{theorem}{Theorem}[section]
\newtheorem{lemma}[theorem]{Lemma}

\newtheorem{proposition}[theorem]{Proposition}
\newtheorem{conjecture}[theorem]{Conjecture}
\newtheorem{definition}[theorem]{Definition}
\newtheorem{example}[theorem]{Example}
\newtheorem{assumption}{Assumption}
\newtheorem{remark}{Remark}

\newcommand{\R}{\mathbb{R}}
\newcommand{\E}{\mathbb{E}}

\newcommand{\Acal}{\mathcal{A}}
\newcommand{\Scal}{\mathcal{S}}
\newcommand{\Ncal}{\mathcal{N}}
\newcommand{\Ccal}{\mathcal{C}}
\DeclareMathOperator{\cov}{cov}
\DeclareMathOperator{\var}{var}
\DeclareMathOperator{\Tr}{Tr}

\begin{document}

\title{Markov Information Processes}
\author{Furkan Sezer\thanks{Texas A\&M University (furkan.sezer@tamu.edu, furkanszr@yahoo.com).}}
\date{\today}

\maketitle

\begin{abstract}
We study information design when a designer with commitment shapes the information of strategically interacting, far-sighted agents whose actions drive a persistent, controlled Markov state. We introduce the Markov Bayes correlated equilibrium (Markov BCE), the controlled-Markov generalisation of the BCE of Bergemann and Morris (2016), characterised by a dynamic obedience condition that adds a continuation-value term to the static one and reduces to it when actions cannot move the state. Recommending actions is without loss; the designer's problem is recursive in the agents' promised continuation utilities and is solved by a set-valued backward-induction algorithm whose optimum exists and lies between the no-disclosure and first-best values. For linear-quadratic-Gaussian payoffs the obedience condition becomes a covariance condition with a modified interaction matrix, and the stationary case reduces to an algebraic Riccati equation. When agents instead learn the transition, we identify the rent an agent earns from a model of the dynamics sharper than the designer anticipates: it is non-negative, zero at the known-dynamics benchmark, and deterred only by building slack into obedience. Under persistent excitation the cumulative rent grows logarithmically as heterogeneous agents' estimates converge. Two worked examples, in congestion and resource coordination, together with a numerical study illustrate the theory.
\end{abstract}


\vspace{0.3cm}

\section{Introduction}
Information design asks how a designer who commits to an information structure can shape the equilibrium behaviour of strategically interacting agents \cite{kamenica2011bayesian, bergemann2019information}. With many agents the relevant solution concept is the Bayes correlated equilibrium (BCE) of \cite{bergemann2016bayes}: by the revelation principle the designer recommends an action profile and the only constraint is \emph{obedience}, that each agent prefers to follow its own recommendation given the posterior the recommendation induces. In the linear-quadratic-Gaussian (LQG) class, obedience reduces to a condition on the second moments of actions and payoff states \cite{Bergemann2013}.

This paper develops a dynamic counterpart in which a persistent state evolves as a controlled Markov process and the agents are \emph{far-sighted}. We call the resulting solution concept the \emph{Markov Bayes correlated equilibrium} (Markov BCE): a generalisation of BCE characterised by a \emph{dynamic obedience condition} in place of the static one. Crucially, a \emph{dynamic revelation principle} holds here, meaning that recommending actions remains without loss of generality in this dynamic, controlled-state environment. The distinction that drives our analysis is between agents who optimize only the current stage and agents who understand that their action today shifts the distribution of tomorrow's state and hence their own future payoffs. The Markov-persuasion literature \cite{wu2022sequential, bacchiocchi2024markov, iyer2023markov} takes the former route: a single stream of myopic receivers, no game among them, and per-stage persuasion. We take the latter. Far-sighted agents make obedience genuinely intertemporal: a deviation must be deterred not only at the stage payoff but also through its continuation consequences. This couples the obedience constraints across time and places the design problem in the recursive, promise-keeping tradition of dynamic incentive design \cite{abreu1990toward, sannikov2008continuous}, while keeping the multi-agent obedience structure of \cite{bergemann2016bayes}. The closest economics antecedent, \cite{makris2023information}, studies information design in multi-stage games with far-sighted players; relative to that work we organise the dynamics around a controlled Markov state and obtain a closed-form covariance characterisation in the LQG class.

The paper has two parts. Part~I (Sections~\ref{sec:static}--\ref{sec:lqg}) is the known-dynamics benchmark: the transition is common knowledge, the designer commits against it, and obedience binds exactly. Part~II (Sections~\ref{sec:learning}--\ref{sec:rentlqg}) perturbs this benchmark by letting agents observe the persistent state and learn the dynamics. Because the transition is controlled, an agent that has learned it better than the designer's committed policy accounts for can deviate to steer the state in its favor; the benchmark of Part~I is precisely the reference point against which this advantage is measured.

The two parts use different informational environments (in Part~I the persistent state is hidden from agents, in Part~II it is observed), and the difference is deliberate. Under known dynamics observability is strategically inert, so hiding $s$ keeps the Part~I benchmark general; it acquires content only once the transition must be learned, because one cannot steer a state one cannot see. Section~\ref{sec:learning} makes this precise and shows the split nests: conditioning the Part~I obedience condition on an observed state recovers the Part~II benchmark as a special case.

\section{Related Work and Positioning}\label{sec:related}
This paper sits at the confluence of information design and Bayes correlated equilibrium, dynamic and Markov persuasion, recursive incentives and promised-utility methods, the identification of linear dynamical systems, and continuous-time stochastic Stackelberg information design. We review each strand and then state how the present framework differs.

\subsection{Information design and Bayes correlated equilibrium.}
The designer-commitment paradigm originates with Bayesian persuasion \cite{kamenica2011bayesian} and its many-player, many-state generalisation, the Bayes correlated equilibrium of \cite{bergemann2016bayes, bergemann2019information}, which lifts correlated equilibrium \cite{aumann1987correlated} to incomplete information and reduces the design problem to an obedience constraint. Information design in games and its public/private signal structure are developed by \cite{mathevet2020information}, and the linear-quadratic-Gaussian specialisation, in which obedience becomes a condition on the second moments of actions and states, is \cite{Bergemann2013}. Our static case (Section~\ref{sec:static}) is exactly this obedience condition; the contribution is the \emph{Markov Bayes correlated equilibrium}, its dynamic, controlled-Markov extension to far-sighted agents, characterised by the dynamic obedience condition.

\subsection{Dynamic and Markov persuasion.}
A growing literature makes persuasion dynamic. \cite{ely2017beeps} discloses a Markov state to a single receiver in continuous time; \cite{renault2017dynamic} characterise optimal dynamic information provision in a Markov environment; \cite{ely2020moving} study dynamic goal-setting; and \cite{doval2020sequential} give a general sequential information-design framework. On the algorithmic side, Markov persuasion processes \cite{wu2022sequential, bacchiocchi2024markov, iyer2023markov} treat a designer facing a stream of receivers in a Markov environment. These analyses share a boundary we cross: the receivers are short-lived (the Markov-persuasion stream), or single with an action that cannot move the state \cite{ely2017beeps, renault2017dynamic}, or single and far-sighted but facing no other strategic player \cite{ely2020moving}, or one of many players who each act once in a designed order \cite{doval2020sequential}. We instead fix a set of far-sighted, mutually strategic agents whose actions move the state, which makes obedience intertemporal and couples the constraints across agents and time. The closest antecedent is \cite{makris2023information}, who generalise Bayes correlated equilibrium to multi-stage games: their \emph{sequential} Bayes correlated equilibrium is characterised by obedience at every history, and they frame the characterisation as a revelation principle for dynamic information design. Our Markov BCE is the specialisation of that concept to a \emph{controlled-Markov} state, where history collapses to the persistent state. The specialisation is what buys the structure absent from the general multi-stage formulation: a recursive, promise-keeping designer problem, a closed-form LQG covariance characterisation, and a stationary infinite-horizon theory. Part~II then departs from the common-knowledge-game premise entirely, asking what an agent gains from learning the state's transition kernel, a question that does not arise when the game is fixed and known.

\subsection{Recursive incentives and promised-utility methods.}
Carrying the agents' promised continuation utilities as state variables is the promise-keeping tradition of repeated and dynamic games \cite{abreu1990toward, spear1987repeated} and its continuous-time counterpart \cite{sannikov2008continuous}. Dynamic mechanism design with evolving private information \cite{pavan2014dynamic} shares the forward-looking incentive structure but allocates through transfers; we have no transfers and design only information, so our binding object is the dynamic obedience condition rather than an envelope or payoff-equivalence formula. A parallel principal--agent line prices the value of information directly: in a continuous-time moral-hazard contract, \cite{hajjej2025value} define the value of the information the principal lacks as the gap between her first-best and second-best value functions, the same first-best-minus-constrained structure our value bounds (Proposition~\ref{prop:valuebounds}) take, though there the instrument is a transfer and here it is disclosure. Section~\ref{sec:designer} casts the designer's problem in this recursive form, with the obedience constraints as the implementability conditions.

\subsection{Learning in games and identification of linear systems.}
Part~II's agents estimate a controlled linear-Gaussian transition, which places the learning analysis in the identification of linear dynamical systems: self-normalised least squares \cite{abbasi2011improved}, finite-sample identification without mixing \cite{simchowitz2018learning}, and the sample complexity of linear-quadratic control \cite{dean2020sample}; the strategic side, agents learning while the designer adapts, connects to learning in games \cite{fudenberg1998theory}. Our use of these tools is specific: the estimation rate feeds, through a Riccati-Lipschitz step, the rent of Proposition~\ref{prop:rent}, yielding the logarithmic-rent Theorem~\ref{thm:regret}. To our knowledge the resulting object, an agent's rent from a superior model of a \emph{designed} environment's dynamics, is new.

\subsection{Continuous-time stochastic Stackelberg information design.}
Finally, the settings of our worked examples come from two continuous-time stochastic Stackelberg control problems in which a leader steers strategic followers through the information structure: \cite{sezer2026power} pairs a committed Gaussian public signal with a Groves transfer for multi-area power-system coordination, while \cite{sezer2026hurricane} steers evacuation zones playing a congestion game via advisory precision and a tiered release schedule, without transfers. In both, the latent state is a filtered jump-diffusion and the leader's distributionally robust problem is an Isaacs equation solved in the viscosity sense. Each is solved on its own terms: stochastic optimal control, not the discrete-time framework developed here. Section~\ref{sec:examples} borrows only their strategic structure, and Part~II's learning problem is a discrete-time, multi-agent analogue of a dynamic-information question raised in the continuous-time persuasion literature \cite{ctpersuasion2024}.

\section{Contributions}
\begin{itemize}
\item We define a \emph{Markov information process} (MKIP): a horizon-$H$, controlled-Markov information-design problem with designer commitment and far-sighted agents (Section~\ref{sec:model}).
\item We introduce the \emph{Markov Bayes correlated equilibrium} (Markov BCE), the controlled-Markov, far-sighted generalisation of BCE, characterised by a \emph{dynamic obedience condition} (Definition~\ref{def:dynobd}) derived from a one-shot-deviation argument; it augments the static Bergemann--Morris obedience condition with a continuation-value differential and reduces to it in the static limit (Proposition~\ref{prop:reduction}). A dynamic revelation principle (Proposition~\ref{prop:revelation}) shows that action recommendations are without loss of generality, justifying the formulation.
\item We formulate the designer's problem recursively, with the agents' promised continuation utilities as endogenous state variables, and give an Abreu--Pearce--Stacchetti set-valued backward-induction algorithm for the designer-optimal Markov BCE (Algorithm~\ref{alg:aps}). We characterise the Markov BCE as a policy/continuation-value fixed point and prove existence (Proposition~\ref{prop:fixedpoint}); under standard regularity the designer's optimum is attained (Proposition~\ref{prop:optexist}) and lies between the no-disclosure and first-best values (Proposition~\ref{prop:valuebounds}).
\item For LQG payoffs (Section~\ref{sec:lqg}) we show dynamic obedience is a covariance condition with a modified interaction matrix $\widetilde\Phi_h$ and a controlled-transition term, with quadratic continuation values (Theorem~\ref{thm:lqg}); the static covariance condition is the action-independent / undiscounted limit, the set-valued recursion collapses to a parametric Riccati recursion, and the infinite-horizon stationary case is a discounted algebraic Riccati fixed point with a contraction-based existence guarantee (Proposition~\ref{prop:stationary}).
\item In Part~II we let agents observe the state and learn the transition. For \emph{general} payoffs and kernels we define the rent an agent earns from a superior model of the dynamics and show it is non-negative, bounded by the agent's misspecification, monotone in the class of models to be deterred, and zero exactly at the known-dynamics benchmark of Part~I (Propositions~\ref{prop:rent_general}, \ref{prop:rentbound}). In the LQG class the rent has a closed form, quadratic in the agent's model error (Proposition~\ref{prop:rent}); deterring it forces the designer to satisfy obedience with a robustness margin (Proposition~\ref{prop:robust}); and under stationary primitives and persistent excitation the cumulative rent is $\tilde O(\log T)$ as heterogeneous agents' estimates of the dynamics converge (Theorem~\ref{thm:regret}).
\item We illustrate the LQG theory with two worked examples (Section~\ref{sec:examples}), drawn from congestion (evacuation) \cite{sezer2026hurricane} and resource-coordination (power) \cite{sezer2026power} settings that are recast as Markov information processes and are related by a single sign of the cross-coupling, together with a numerical illustration comparing utilitarian and worst-case (maximin) designs (Section~\ref{sec:numerics}).
\end{itemize}

\begin{center}\rule{0.6\linewidth}{0.4pt}\\[4pt]{\large\bfseries Part I: Markov Information Processes with Known Dynamics}\\[2pt]\rule{0.6\linewidth}{0.4pt}\end{center}

\section{Static Case: Bayes Correlated Equilibrium}\label{sec:static}
Fix a finite set of players $\Ncal=\{1,\dots,n\}$, action sets $\Acal_i$ with profiles $a=(a_i)_{i\in\Ncal}\in\Acal=\prod_i\Acal_i$, a payoff state $\gamma\in\Gamma$ with common prior $\mu\in\Delta(\Gamma)$, and payoffs $u_i(\gamma,a)$. A \emph{decision rule} is a map $\sigma:\Gamma\to\Delta(\Acal)$; by the revelation principle we read $\sigma(a\mid\gamma)$ as the probability of recommending profile $a$ in state $\gamma$, with each agent privately told its own component $a_i$ (the dynamic counterpart of this reduction is Proposition~\ref{prop:revelation}).

\begin{definition}[Obedience / BCE, \cite{bergemann2016bayes}]\label{def:bce}
The decision rule $\sigma$ is a \emph{Bayes correlated equilibrium} if for every $i\in\Ncal$, every recommended $a_i$ occurring with positive probability, and every alternative $a_i'\in\Acal_i$,
\begin{equation}\label{eq:static_obd}
\sum_{\gamma\in\Gamma}\sum_{a_{-i}\in\Acal_{-i}}\mu(\gamma)\,\sigma(a_i,a_{-i}\mid\gamma)\,\big[u_i(\gamma,a_i,a_{-i})-u_i(\gamma,a_i',a_{-i})\big]\ \ge\ 0 .
\end{equation}
\end{definition}

Condition \eqref{eq:static_obd} is the obedience condition of \cite{bergemann2016bayes}: holding fixed the joint law of the state and the recommendations to others, following the recommendation is optimal. It is the ground on which the rest of the paper is built; the dynamic obedience condition given at Definition \ref{def:dynobd} generalises \eqref{eq:static_obd}.

\section{Markov Information Processes}\label{sec:model}
A \emph{Markov information process} (MKIP) augments the static game with a persistent state and a controlled transition, and lets agents be far-sighted over a horizon $H$ with discount factor $\delta\in(0,1]$.

\textbf{Primitives.} At each stage $h\in[H]:=\{1,\dots,H\}$ there are: a persistent Markov state $s\in\Scal$, observed by the designer; a transient payoff state $\gamma\in\Gamma$ with prior $\mu_h\in\Delta(\Gamma)$; stage payoffs $u_{h,i}:\Scal\times\Gamma\times\Acal\to\R$; a designer objective $v_h:\Scal\times\Gamma\times\Acal\to\R$; and a transition kernel $P_h:\Scal\times\Gamma\times\Acal\to\Delta(\Scal)$ depending on the \emph{realised} actions.

\textbf{Timing within stage $h$.}
\begin{enumerate}
\item The persistent state $s_h$ is given and $\gamma_h\sim\mu_h$ is drawn.
\item The designer, having committed at the outset to a policy $\sigma=(\sigma_h)_{h\in[H]}$ with $\sigma_h:\Scal\times\Gamma\to\Delta(\Acal)$, draws a profile $a_h\sim\sigma_h(\cdot\mid s_h,\gamma_h)$ and privately recommends $a_{h,i}$ to each agent $i$.
\item Each agent chooses an action; payoffs $u_{h,i}(s_h,\gamma_h,\cdot)$ and the designer objective $v_h(s_h,\gamma_h,\cdot)$ accrue at the realised profile.
\item The next state is drawn from the realised profile, $s_{h+1}\sim P_h(\cdot\mid s_h,\gamma_h,a_h)$.
\end{enumerate}

In Part~I, agents observe their own recommendation (and the public calendar $h$); they need not observe $s_h$ or $\gamma_h$, about which the recommendation is informative. The designer commits to $\sigma$ and cannot react to an unobserved within-stage deviation; deviations propagate only through the realised next state $s_{h+1}$, which the designer observes and on which its continuation policy depends.

\textbf{Far-sighted agents.} Each agent $i$ ranks outcomes by the expected discounted sum
\begin{equation}\label{eq:pref}
\E\Big[\textstyle\sum_{h=1}^{H}\delta^{h-1}u_{h,i}(s_h,\gamma_h,a_h)\Big].
\end{equation}
In particular an agent that contemplates disobeying at stage $h$ accounts for the effect of its action on $s_{h+1}$ and hence on its own continuation payoff. This is the sole behavioural departure from the myopic Markov-persuasion model, and it is what makes obedience dynamic.

\section{Dynamic Obedience}\label{sec:dynamic}
Fix a policy $\sigma$. Under full obedience the process is Markov in $s$, so define agent $i$'s \emph{continuation value} from stage $h$ at persistent state $s$,
\begin{equation}\label{eq:contval}
W_{h,i}(s)\ :=\ \E\Big[\textstyle\sum_{\tau=h}^{H}\delta^{\tau-h}\,u_{\tau,i}(s_\tau,\gamma_\tau,a_\tau)\ \Big|\ s_h=s\Big],\qquad W_{H+1,i}\equiv 0,
\end{equation}
where the expectation is taken under $\gamma_\tau\sim\mu_\tau$, $a_\tau\sim\sigma_\tau(\cdot\mid s_\tau,\gamma_\tau)$, $s_{\tau+1}\sim P_\tau(\cdot\mid s_\tau,\gamma_\tau,a_\tau)$ with every agent obeying.

Because payoffs are bounded and the horizon is finite, a recommendation policy is incentive compatible for far-sighted agents if and only if no single-stage deviation is profitable (the one-shot-deviation principle). At stage $h$, persistent state $s$, an agent recommended $a_i$ that contemplates playing $a_i'$ instead obtains, in conditional expectation over $(\gamma,a_{-i})$ given $a_i$,
\begin{equation}\label{eq:Q}
Q_{h,i}(s,a_i,a_i')\ :=\ \E\Big[\,u_{h,i}(s,\gamma,a_i',a_{-i})\ +\ \delta\,\E_{s'\sim P_h(\cdot\mid s,\gamma,a_i',a_{-i})}\big[W_{h+1,i}(s')\big]\ \Big|\ a_i\Big],
\end{equation}
where the inner expectation reflects that the deviation changes the law of $s_{h+1}$, after which the agent reverts to obedience.

\begin{definition}[Dynamic obedience]\label{def:dynobd}
The policy $\sigma$ satisfies \emph{dynamic obedience}, and is a \emph{Markov Bayes correlated equilibrium} (Markov BCE) of the MKIP, if for every $i\in\Ncal$, stage $h$, state $s$, recommended $a_i$ (positive probability), and alternative $a_i'$,
\begin{equation}\label{eq:dynobd}
\begin{aligned}
\sum_{\gamma,\,a_{-i}}\mu_h(\gamma)\,\sigma_h(a_i,a_{-i}\mid s,\gamma)\Big[&\underbrace{u_{h,i}(s,\gamma,a_i,a_{-i})-u_{h,i}(s,\gamma,a_i',a_{-i})}_{\text{stage incentive}}\\
&{}+\ \delta\,\underbrace{\Delta^{i}_{h}(s,\gamma,a_i,a_{-i};a_i')}_{\text{continuation incentive}}\Big]\ \ge\ 0,
\end{aligned}
\end{equation}
where the continuation-value differential is
\begin{equation}\label{eq:Delta}
\begin{aligned}
\Delta^{i}_{h}(s,\gamma,a_i,a_{-i};a_i')\ :=\ \E_{s'\sim P_h(\cdot\mid s,\gamma,a_i,a_{-i})}\!\big[W_{h+1,i}(s')\big]
-\ \E_{s'\sim P_h(\cdot\mid s,\gamma,a_i',a_{-i})}\!\big[W_{h+1,i}(s')\big].
\end{aligned}
\end{equation}
\end{definition}

Equation \eqref{eq:dynobd} is the obedience condition \eqref{eq:static_obd} with one new object: the term $\delta\,\Delta^i_h$, which prices the future. Obeying and disobeying lead to different distributions over $s_{h+1}$, hence to different continuation values $W_{h+1,i}$; a far-sighted agent internalises this, and the designer must compensate for it. The constraint at stage $h$ therefore depends, through $W_{h+1,i}$, on the entire continuation policy: the obedience constraints are coupled forward in time.

\begin{proposition}[Reduction to the static case]\label{prop:reduction}
Suppose (i) $H=1$; or (ii) the transition is action-independent, $P_h(\cdot\mid s,\gamma,a)=P_h(\cdot\mid s,\gamma)$ for all $a$; or (iii) $\delta=0$. Then $\Delta^i_h\equiv 0$ and \eqref{eq:dynobd} coincides, stage by stage, with the static Bergemann--Morris obedience condition \eqref{eq:static_obd}.
\end{proposition}
\begin{proof}
Under (i) $W_{H+1,i}\equiv0$; under (ii) the two laws of $s'$ in \eqref{eq:Delta} are identical; under (iii) the continuation term is multiplied by zero. In each case $\Delta^i_h\equiv0$, and \eqref{eq:dynobd} becomes $\sum_{\gamma,a_{-i}}\mu_h\,\sigma_h[u_{h,i}(s,\gamma,a_i,a_{-i})-u_{h,i}(s,\gamma,a_i',a_{-i})]\ge0$, which is \eqref{eq:static_obd}.
\end{proof}

\section{The Designer's Problem}\label{sec:designer}
The designer's problem is
\begin{equation}\label{eq:designer}
\max_{\sigma}\ \ \E\Big[\sum_{h=1}^{H}\delta^{h-1}v_h(s_h,\gamma_h,a_h)\Big]\quad\text{subject to dynamic obedience \eqref{eq:dynobd}},
\end{equation}
imposed at every stage, state, and recommendation. Restricting the designer's instrument to an action recommendation $\sigma_h:\Scal\times\Gamma\to\Delta(\Acal)$ rather than an abstract private signal is without loss: it is the dynamic counterpart of the revelation principle for Bayes correlated equilibrium \cite{makris2023information}, specialised to the controlled-Markov setting.

\begin{proposition}[Dynamic revelation principle]\label{prop:revelation}
Consider any dynamic information structure in which, at each stage $h$, the designer commits to a private-signalling kernel $\pi_h:\Scal\times\Gamma\to\Delta(M)$ sending agent $i$ a message $m_{h,i}$, and the agents play a perfect Bayesian equilibrium of the induced game under the discounted preferences \eqref{eq:pref}, with deviations unmonitored and propagating only through the realised next state. Then there is a direct recommendation policy $\sigma=(\sigma_h)$, $\sigma_h:\Scal\times\Gamma\to\Delta(\Acal)$, that induces the same joint law of $(s_h,\gamma_h,a_h)$ at every stage (hence the same designer value and agent payoffs), and under which obeying the recommendation is an equilibrium, i.e.\ $\sigma$ satisfies dynamic obedience \eqref{eq:dynobd}. Conversely every dynamically obedient $\sigma$ is induced by an information structure, namely itself. The designer's problem is therefore without loss the optimization \eqref{eq:designer} over dynamically obedient recommendation policies.
\end{proposition}
\begin{proof}
Fix the information structure $\pi$ and an equilibrium with action strategies mapping message histories to actions, and define the recommendation policy by the equilibrium action law,
\[
\sigma_h(a\mid s,\gamma)\ :=\ \Pr\big[a_h=a\,\big|\,s_h=s,\ \gamma_h=\gamma\big]
\]
under $\pi$ and the equilibrium. Because the transition $P_h$ depends on the realised $(s_h,\gamma_h,a_h)$ alone and not on messages, recommending $a_h\sim\sigma_h$ and having agents obey reproduces, by induction on $h$, the same joint law of $(s_h,\gamma_h,a_h)$ as $\pi$; the designer value and all agent payoffs coincide.

For obedience, when agent $i$ is recommended $a_{h,i}$ under $\sigma$ its posterior over $(s_h,\gamma_h,a_{h,-i})$ is the pooling of the posteriors it held, under $\pi$, at every message that led the equilibrium to play $a_{h,i}$. Two features make the recommended action optimal under this pooled posterior. First, by the model's monitoring structure a within-stage deviation is unobserved and affects the future only through the realised $s_{h+1}$, so the deviator's continuation value is the designer's committed continuation evaluated at $s_{h+1}$, the same function of the deviation under $\sigma$ as under $\pi$. The one-shot-deviation comparison at stage $h$ is therefore identical under the two: the change in current payoff plus $\delta$ times the continuation change transmitted through $s_{h+1}$, which is exactly the differential $\Delta^i_h$ of \eqref{eq:dynobd}. Second, the recommended action was optimal at each pooled message in the equilibrium under $\pi$; since the stage-plus-continuation payoff is affine in the posterior, an action optimal at every message in the pool is optimal at their average. Hence obeying $a_{h,i}$ is optimal under the recommendation's posterior for every $i$ and $h$, which is \eqref{eq:dynobd}. The converse is immediate.
\end{proof}

\begin{remark}[Scope]
The principle holds within the information-design class the model defines: the designer designs information and cannot monitor or contract on within-stage actions. If the designer could condition its continuation on \emph{observed} within-stage deviations, richer history-dependent schemes could enlarge the implementable set and action recommendations would no longer be without loss, a monitoring/mechanism-design problem outside the present scope.
\end{remark}

Because the stage-$h$ constraints in \eqref{eq:designer} involve the continuation values $W_{h+1,i}$, the problem does not separate into independent per-stage programs. The standard device is to carry the agents' \emph{promised continuation utilities} $w=(w_i)_{i\in\Ncal}$ as endogenous state variables \cite{abreu1990toward, sannikov2008continuous}.

Let $J_h(s,w)$ be the largest discounted designer objective attainable from stage $h$ at persistent state $s$ when each agent $i$ is promised continuation utility $w_i$ and obedience holds. Then
\begin{equation}\label{eq:bellman}
J_h(s,w)=\max_{\sigma_h,\ \{w'(\cdot)\}}\ \E\Big[v_h(s,\gamma,a)+\delta\,J_{h+1}\big(s',w'(s')\big)\Big]
\end{equation}
subject to, for all $i,a_i,a_i'$, the dynamic obedience condition \eqref{eq:dynobd} written with $W_{h+1,i}(\cdot)$ replaced by the continuation-utility map $w'_i(\cdot)$, and the promise-keeping identities
\begin{equation}\label{eq:promise}
w_i=\E\big[u_{h,i}(s,\gamma,a)+\delta\,w'_i(s')\big],\qquad i\in\Ncal,
\end{equation}
with terminal condition $J_{H+1}(s,\cdot)=0$. A policy is optimal if and only if it attains \eqref{eq:bellman} at every stage given a self-consistent profile of continuation-utility maps. The recursion makes explicit that the agents' continuation values are co-state variables of the design problem: the price of the dynamic incentive term $\delta\Delta^i_h$ is the shadow value of the promise-keeping constraint \eqref{eq:promise}.

The recursion \eqref{eq:bellman} is solved by an Abreu--Pearce--Stacchetti (APS) set-valued backward induction \cite{abreu1990toward}. Because the stage-$h$ obedience and promise-keeping constraints depend on the continuation only through the promised-utility profile $w'(\cdot)$, the designer need not track the entire continuation policy but only, at each state, the \emph{set} of promised-utility profiles an obedient continuation can deliver. Write $\Ccal_h(s)\subseteq\R^{\Ncal}$ for that set. It is computed backward by Algorithm~\ref{alg:aps}: the designer's optimal value is $\max_{w\in\Ccal_1(s_1)}J_1(s_1,w)$, and the optimal policy is recovered by unrolling the stored maximisers forward.

\begin{algorithm}[t]
\caption{Designer-optimal Markov BCE: APS-style promised-utility backward induction}\label{alg:aps}
\begin{algorithmic}[1]
\State \textbf{Initialize} $\Ccal_{H+1}(s)\gets\{0\}$ and $J_{H+1}(s,\cdot)\gets 0$ for all $s$.
\For{$h=H$ \textbf{down to} $1$}
  \For{each persistent state $s$}
    \State Form the admissible set of promised utilities
    \Statex \quad $\Ccal_h(s)\gets\big\{\,w\in\R^{\Ncal}:\ \exists\ \sigma_h(\cdot\mid s,\gamma)\ \text{and a map}\ w'(\cdot)\ \text{with}\ w'(s')\in\Ccal_{h+1}(s')\ \forall s',$
    \Statex \quad\qquad such that $\sigma_h$ satisfies dynamic obedience \eqref{eq:dynobd} with continuation $w'$, and promise-keeping \eqref{eq:promise} holds $\big\}$.
    \For{each $w\in\Ccal_h(s)$}
      \State $J_h(s,w)\gets\displaystyle\max_{\sigma_h,\,w'(\cdot)}\ \E\big[v_h(s,\gamma,a)+\delta\,J_{h+1}(s',w'(s'))\big]$ over the $(\sigma_h,w')$ admitting $w$.
    \EndFor
  \EndFor
\EndFor
\State \textbf{Output} optimal value $\max_{w\in\Ccal_1(s_1)}J_1(s_1,w)$ and the attaining policy $\sigma^\star$ by forward unrolling the stored maximisers.
\end{algorithmic}
\end{algorithm}

In general the sets $\Ccal_h(s)$ are computed by the standard APS monotone outer-approximation iteration. In the LQG class they admit a finite parametrisation: the continuation enters only through the quadratic coefficient $\Pi_{h+1,i}$ of Lemma~\ref{lem:quadratic}, obedience becomes the covariance condition of Theorem~\ref{thm:lqg}, and the set-valued recursion of Algorithm~\ref{alg:aps} reduces to the parametric Riccati recursion \eqref{eq:riccati} of Lemma~\ref{lem:quadratic}; in the infinite-horizon stationary case it becomes the discounted algebraic Riccati fixed point of Proposition~\ref{prop:stationary}, solved by value iteration.

\begin{remark}[Computing the sets in general]\label{rem:aps_comp}
Outside structured classes the sets $\Ccal_h(s)$ are continuous set-valued objects and exact computation is not tractable. Under a public randomisation device each $\Ccal_h(s)$ is convex, and the standard remedy applies: inner and outer polytope approximations of the APS operator converge monotonically to $\Ccal_h(s)$ and bracket the optimal value \cite{abreu1990toward, judd2003computing}. The LQG collapse above is what turns Algorithm~\ref{alg:aps} from a conceptual scheme into an exact procedure; its complexity in that class is given in Proposition~\ref{prop:complexity}.
\end{remark}

The designer's optimal value sits between two policy-free benchmarks, which also bound what commitment buys.

\begin{proposition}[Value bounds and commitment monotonicity]\label{prop:valuebounds}
Let $J^\star(s_1)=\max_{w\in\Ccal_1(s_1)}J_1(s_1,w)$ be the designer-optimal value. Let $V^{\mathrm{bab}}(s_1)$ be the designer's value when it recommends, in every stage, a Bayes--Nash equilibrium of the stage game under the prior (no state-contingent disclosure), and let $V^{\mathrm{fb}}(s_1)$ be the value of the relaxed problem in which the obedience constraints \eqref{eq:dynobd} are dropped and the designer chooses the action profile directly. Then
\[
V^{\mathrm{bab}}(s_1)\ \le\ J^\star(s_1)\ \le\ V^{\mathrm{fb}}(s_1),
\]
and $J^\star$ is non-decreasing in the designer's feasible recommendation set: a richer message space or weaker information constraint weakly raises $J^\star$. The gaps $J^\star-V^{\mathrm{bab}}\ge0$ and $V^{\mathrm{fb}}-J^\star\ge0$ measure, respectively, the value of information design and the price of obedience.
\end{proposition}
\begin{proof}
A stationary Bayes--Nash recommendation is dynamically obedient (each agent best-responds stage by stage, continuation included), hence a Markov BCE and feasible in \eqref{eq:bellman}, giving $V^{\mathrm{bab}}\le J^\star$. Dropping \eqref{eq:dynobd} enlarges the feasible set of \eqref{eq:bellman}, so its maximum cannot fall: $J^\star\le V^{\mathrm{fb}}$. Both bounds are instances of the same fact, that the value of \eqref{eq:bellman} is monotone in the feasible set, which also gives monotonicity of $J^\star$ in the designer's commitment.
\end{proof}

\subsection{The policy / continuation-value fixed point}\label{sec:fixedpoint}
Dynamic obedience \eqref{eq:dynobd} at stage $h$ is stated in terms of the continuation values $W_{h+1,i}$, which are themselves generated by the policy through \eqref{eq:contval}. A Markov BCE is therefore a joint solution of these two dependencies, which we record as a fixed point.

\begin{proposition}[Markov BCE as a fixed point; existence]\label{prop:fixedpoint}
For a policy $\sigma$ let $\mathcal{T}(\sigma)=(W_{h,i})_{h,i}$ be the continuation values \eqref{eq:contval} it induces, and let $\mathcal{O}(\mathcal{W})$ be the set of policies satisfying \eqref{eq:dynobd} when each $W_{h+1,i}$ is taken from $\mathcal{W}$. A policy is a Markov Bayes correlated equilibrium iff $\sigma\in\mathcal{O}(\mathcal{T}(\sigma))$. In finite horizon this fixed point is resolved by backward induction, and a Markov BCE exists whenever the stage obedient set is nonempty at every stage, in particular whenever the stage game admits a Bayes--Nash equilibrium, which for the LQG payoff \eqref{utility} holds under $\Phi+\Phi^\top\succ0$.
\end{proposition}
\begin{proof}
The equivalence is the definition: $\sigma$ obeys \eqref{eq:dynobd} against its own induced continuation values iff $\sigma\in\mathcal{O}(\mathcal{T}(\sigma))$. The dependence is acyclic in finite horizon: $W_{h+1,i}$ depends only on $\sigma_{h+1},\dots,\sigma_H$, never on $\sigma_h$. Hence backward induction applies: with $W_{H+1,i}\equiv0$ choose $\sigma_H$ in the (static) stage-$H$ obedient set, which fixes $W_{H,\cdot}$; given $W_{h+1,\cdot}$ choose any $\sigma_h$ satisfying \eqref{eq:dynobd}, a system of linear inequalities in $\sigma_h$ decoupled from earlier stages, which fixes $W_{h,\cdot}$; iterate to $h=1$. The resulting $\sigma$ satisfies \eqref{eq:dynobd} at every stage. Nonemptiness of each stage set holds whenever the stage game has a Bayes--Nash equilibrium (recommend it conditional on $\gamma$); for \eqref{utility} with $\Phi+\Phi^\top\succ0$ the quadratic stage game is strictly concave with a unique such equilibrium.
\end{proof}

Proposition~\ref{prop:fixedpoint} gives existence of \emph{a} Markov BCE; the designer wants the value-maximising one, and Algorithm~\ref{alg:aps} returns its value provided that maximum is attained. Standard regularity makes it so.

\begin{proposition}[Existence of an optimal Markov BCE]\label{prop:optexist}
Suppose the payoff-state space $\Gamma$ and the action space $\Acal$ are compact metric, the stage payoffs $u_{h,i}$ and the designer payoff $v_h$ are bounded and continuous, and each transition kernel $P_h$ is weakly (Feller) continuous. Then for every $h$ and $s$ the set $\Ccal_h(s)$ of implementable promised-utility profiles is compact and the value $J_h(s,\cdot)$ is upper semicontinuous, the designer's problem $\max_{w\in\Ccal_1(s_1)}J_1(s_1,w)$ attains its maximum, and an optimal Markov BCE policy exists and is the one Algorithm~\ref{alg:aps} computes.
\end{proposition}
\begin{proof}
Recommendation policies $\sigma_h(\cdot\mid s,\gamma)$ are probability measures on the compact $\Acal$, a weak-$*$ compact set (Prokhorov). Backward induction: $\Ccal_{H+1}(s)=\{0\}$ is compact and $J_{H+1}\equiv0$ is continuous. Assume $\Ccal_{h+1}(\cdot)$ compact-valued and $J_{h+1}$ upper semicontinuous. The stage-$h$ constraints (dynamic obedience \eqref{eq:dynobd} and promise-keeping \eqref{eq:promise}) are, by continuity of $u_{h,i}$ and Feller continuity of $P_h$, closed conditions on $(\sigma_h,w'(\cdot))$, so the feasible set is a closed, hence compact, subset of the compact policy space; $\Ccal_h(s)$ is its image under the continuous promise map \eqref{eq:promise} and is therefore compact. The stage objective $\E[v_h+\delta J_{h+1}(s',w'(s'))]$ is upper semicontinuous in $(\sigma_h,w')$ (Feller continuity and the induction hypothesis), so by Berge's maximum theorem $J_h(s,\cdot)$ is upper semicontinuous and the stage argmax is nonempty and compact-valued. Iterating to $h=1$, $\Ccal_1(s_1)$ is compact and $J_1(s_1,\cdot)$ upper semicontinuous, so the maximum is attained (Weierstrass); a measurable selection of stage maximisers yields the optimal policy, which is exactly the forward unrolling in Algorithm~\ref{alg:aps}.
\end{proof}

\begin{remark}[Designer selection and uniqueness]
The fixed point is generically not unique: like static BCE, dynamic obedience defines a \emph{set} of policies. The designer selects within it by value through the recursion \eqref{eq:bellman}, so the optimal policy is the value-maximising fixed point. By contrast, the continuation values \emph{along a fixed policy} are unique, being the backward evaluation of \eqref{eq:contval}. In the infinite-horizon stationary case the acyclic structure is lost and the fixed point becomes a genuine algebraic one (Section~\ref{sec:stationary}).
\end{remark}

\section{LQG Markov Information Processes}\label{sec:lqg}
We specialise to the linear-quadratic-Gaussian class. In an LQG game, player $i$'s payoff is quadratic,
\begin{equation}\label{utility}
u_{i}(s,\gamma,a) = - \Phi_{i,i}a_{i}^{2} - 2 \sum_{j \neq i}\Phi_{i,j}a_{i}a_{j} + 2(\gamma_{i}+s_{i})a_{i}+d_{i}(a_{-i}, \gamma),
\end{equation}
where the $\Phi_{i,j}$, $i,j\in\mathcal{N}$, are real coefficients with $\Phi_{i,i}>0$, and $d_{i}(a_{-i},\gamma)$ is an arbitrary function of the opponents' actions $a_{-i}\equiv\{a_{j}\}_{j\neq i}$ and the payoff state $\gamma$, with $a\in\mathcal{A}\equiv\mathbb{R}^{n}$ and $\gamma\in\Gamma\equiv\mathbb{R}^{n}$. We collect the coefficients in $\Phi=[\Phi_{i,j}]\in\mathbb{R}^{n\times n}$. The payoff is quadratic in $a_i$ but need not be quadratic in the others' actions or the payoff state, via the term $d_{i}(a_{-i},\gamma)$; this term cannot be influenced by player $i$ and so does not affect its best response. We focus on scalar actions, $a_i\in\mathbb{R}$. We write $s_i$ for the Markov state of agent $i$, forming the $i$-th component of $s\in\mathbb{R}^{n}$.

We pair \eqref{utility} with a controlled linear-Gaussian transition,
\begin{equation}\label{eq:transition}
s_{h+1}=A_h s_h+B_h a_h+C_h\gamma_h+\varepsilon_h,\qquad \varepsilon_h\sim\mathcal{N}(0,\Sigma_\varepsilon),\ \ \varepsilon_h\perp(s_h,\gamma_h,a_h),
\end{equation}
with $A_h,B_h,C_h\in\R^{n\times n}$, and we write $B_h^{(i)}$ for the $i$-th column of $B_h$ (the channel through which $a_i$ moves the state). The prior on $(\gamma_h,s_h)$ and the recommendation policy are jointly Gaussian, so that $(\gamma_h,s_h,a_h)$ is jointly Gaussian under the designer's policy and conditional expectations are affine. Throughout, $\E$ and $\cov$ are taken under this joint law at the stage in question.

\begin{lemma}[Quadratic continuation values]\label{lem:quadratic}
Under \eqref{utility}--\eqref{eq:transition} and a linear-Gaussian recommendation policy, each agent's continuation value \eqref{eq:contval} is quadratic,
\begin{equation}\label{eq:Wquad}
W_{h,i}(s)=s^\top\Pi_{h,i}\,s+\pi_{h,i}^\top s+c_{h,i},\qquad \Pi_{H+1,i}=0,\ \pi_{H+1,i}=0,\ c_{H+1,i}=0,
\end{equation}
where $(\Pi_{h,i},\pi_{h,i},c_{h,i})$ solve the backward recursion obtained by substituting the policy-induced affine law of $a_h$ in $(\gamma_h,s_h)$ and the transition \eqref{eq:transition} into $W_{h,i}(s)=\E[u_{h,i}(s,\gamma,a)+\delta W_{h+1,i}(s')\mid s_h=s]$ and collecting terms by degree in $s$. In particular the quadratic coefficient satisfies the finite-horizon Riccati recursion
\begin{equation}\label{eq:riccati}
\Pi_{h,i}\ =\ \delta\,A_h^\top\Pi_{h+1,i}A_h+R_{h,i},\qquad \Pi_{H+1,i}=0,
\end{equation}
where $R_{h,i}$ gathers the stage-payoff and policy-dependent quadratic terms; the recursion is \emph{parametric} in the recommendation policy through $R_{h,i}$.
\end{lemma}
\begin{proof}
Backward induction. $W_{H+1,i}\equiv0$ is \eqref{eq:Wquad} with zero coefficients. If $W_{h+1,i}$ is quadratic, then since $a_h$ is affine in $(\gamma_h,s_h)$ under the policy and $s'$ is affine in $(s_h,\gamma_h,a_h)$ by \eqref{eq:transition}, the integrand $u_{h,i}+\delta W_{h+1,i}(s')$ is quadratic in $(s_h,\gamma_h,a_h,\varepsilon_h)$; taking the Gaussian expectation given $s_h=s$ leaves a quadratic in $s$ (the noise contributes only the constant $\delta\,\Tr(\Pi_{h+1,i}\Sigma_\varepsilon)$). Matching degrees gives the stated recursion, with the $\delta A_h^\top\Pi_{h+1,i}A_h$ term coming from the $s$-linear part of $\bar s'=\E[s'\mid s]$.
\end{proof}

\begin{theorem}[Dynamic LQG obedience]\label{thm:lqg}
Consider an LQG MKIP with payoff \eqref{utility} and transition \eqref{eq:transition}, and suppose $\Phi+\Phi^\top$ is positive definite. Then dynamic obedience \eqref{eq:dynobd} holds at stage $h$ if and only if, for every $i\in\Ncal$,
\begin{equation}\label{eq:lqg_dynobd}
\begin{aligned}
\sum_{j\in\Ncal}\widetilde\Phi_{h,ij}\,\cov(a_i,a_j)\ =\ \cov(a_i,\gamma_i)+\cov(a_i,s_i)
\quad+\ \delta\,(B_h^{(i)})^\top\Pi_{h+1,i}\big[A_h\,\cov(a_i,s)+C_h\,\cov(a_i,\gamma)\big],
\end{aligned}
\end{equation}
where $\cov(a_i,s)=(\cov(a_i,s_k))_{k\in\Ncal}$, $\cov(a_i,\gamma)=(\cov(a_i,\gamma_k))_{k\in\Ncal}$, the continuation matrices $\Pi_{h+1,i}$ are those of Lemma~\ref{lem:quadratic}, and the \emph{effective interaction matrix} is
\begin{equation}\label{eq:Htilde}
\widetilde\Phi_{h,ij}\ :=\ \Phi_{ij}\ -\ \delta\,(B_h^{(i)})^\top\Pi_{h+1,i}\,B_h^{(j)} .
\end{equation}
When the transition is action-independent ($B_h=0$) or $\delta=0$, \eqref{eq:lqg_dynobd} reduces to the static Bergemann--Morris condition
\begin{equation}\label{eq:static_cov}
\sum_{j\in\Ncal}\Phi_{ij}\,\cov(a_i,a_j)=\cov(a_i,\gamma_i)+\cov(a_i,s_i),\qquad i\in\Ncal .
\end{equation}
\end{theorem}
\begin{proof}
Since $\Phi_{ii}>0$ the stage payoff is strictly concave in $a_i$, and by \eqref{eq:transition} and Lemma~\ref{lem:quadratic} the continuation term $\delta\,\E[W_{h+1,i}(s')\mid a_i]$ is quadratic in the agent's own action; $\Phi+\Phi^\top\succ0$ guarantees the resulting stage best response is well defined. Hence dynamic obedience \eqref{eq:dynobd} holds iff the recommended action satisfies the first-order condition in conditional expectation given the recommendation $a_i$. Writing $\bar s'=A_h s+B_h a+C_h\gamma$ for the conditional mean next state and differentiating $u_{h,i}+\delta W_{h+1,i}(\bar s')$ in $a_i$,
\begin{equation}\label{eq:foc}
\E\Big[-2\Phi_{ii}a_i-2\!\sum_{j\neq i}\Phi_{ij}a_j+2(\gamma_i+s_i)+2\delta(B_h^{(i)})^\top\Pi_{h+1,i}\bar s'+\delta(B_h^{(i)})^\top\pi_{h+1,i}\ \Big|\ a_i\Big]=0 .
\end{equation}
Substituting $\bar s'=A_h s+\sum_k B_h^{(k)}a_k+C_h\gamma$ and collecting the action terms, the coefficient on $a_j$ is $-2(\Phi_{ij}-\delta(B_h^{(i)})^\top\Pi_{h+1,i}B_h^{(j)})=-2\widetilde\Phi_{h,ij}$. Dividing \eqref{eq:foc} by $2$ gives, in conditional expectation,
\begin{equation}\label{eq:foc2}
\begin{aligned}
\sum_{j\in\Ncal}\widetilde\Phi_{h,ij}\,\E[a_j\mid a_i]&=\E[\gamma_i\mid a_i]+\E[s_i\mid a_i]
+\delta(B_h^{(i)})^\top\Pi_{h+1,i}\big(A_h\E[s\mid a_i]+C_h\E[\gamma\mid a_i]\big)
+\tfrac{\delta}{2}(B_h^{(i)})^\top\pi_{h+1,i}.
\end{aligned}
\end{equation}
Multiplying \eqref{eq:foc2} by the centered recommendation $(a_i-\E[a_i])$ and taking expectations, the tower property gives $\E[(a_i-\E a_i)\,\E[X\mid a_i]]=\cov(a_i,X)$; the constant term drops, and we obtain \eqref{eq:lqg_dynobd}. Conversely, under joint normality the conditional expectations in \eqref{eq:foc2} are affine, so \eqref{eq:lqg_dynobd} implies \eqref{eq:foc2} and hence \eqref{eq:dynobd}. Setting $B_h=0$ kills both the $\widetilde\Phi$ correction and the bracketed term, and $\delta=0$ does likewise, leaving \eqref{eq:static_cov}, the second-moment characterisation of \cite{Bergemann2013}.
\end{proof}

\begin{remark}[What far-sightedness adds]
Relative to the static condition \eqref{eq:static_cov}, Theorem~\ref{thm:lqg} shows that far-sighted agents change the obedience condition in two ways. First, the strategic-interaction matrix is shifted from $\Phi$ to $\widetilde\Phi_h$: the controlled transition makes agents' continuation values interact through the state even where their stage payoffs do not. Second, a new term couples the recommendation to the persistent and payoff states through the transition kernel. Both vanish when the agent's action cannot move the state ($B_h=0$), yielding the myopic-receiver case.
\end{remark}

\subsection{Infinite-horizon stationary case}\label{sec:stationary}
Take the primitives time-homogeneous ($A_h\equiv A$, $B_h\equiv B$, $C_h\equiv C$, $\mu_h\equiv\mu$), let $H\to\infty$, and seek a stationary policy with stationary continuation value $W_i(s)=s^\top\Pi_i s+\pi_i^\top s+c_i$. The backward recursion of Lemma~\ref{lem:quadratic} becomes the algebraic fixed point
\begin{equation}\label{eq:are}
\Pi_i\ =\ \delta\,A^\top\Pi_i A\ +\ R_i(\Xi;\Pi_i),
\end{equation}
where $\Xi$ collects the stationary second moments of $(\gamma,s,a)$ under the policy and $R_i$ is the quadratic-in-$s$ part of the stage return. Because the obedient policy's own-curvature is renormalised to $\widetilde\Phi_{ii}=\Phi_{ii}-\delta(B^{(i)})^\top\Pi_iB^{(i)}$ by Theorem~\ref{thm:lqg}, the policy that $R_i$ is evaluated at depends on $\Pi_i$, so \eqref{eq:are} is a genuine discounted \emph{algebraic Riccati} equation rather than a linear Lyapunov equation. A stationary Markov BCE is a joint fixed point $(\Xi^\star,\Pi^\star)$ of \eqref{eq:are} and the stationary obedience condition \eqref{eq:lqg_dynobd} evaluated with $\widetilde\Phi(\Pi^\star)$.

\begin{proposition}[Existence of the stationary fixed point]\label{prop:stationary}
Suppose $\delta\,\|A\|^2<1$. Then for each fixed admissible stationary policy (hence fixed $R_i$), the map $\mathcal{L}(\Pi)=\delta A^\top\Pi A+R_i$ is a contraction of modulus $\delta\|A\|^2$ in the spectral norm, with unique fixed point the convergent Neumann series $\Pi_i=\sum_{k\ge0}\delta^k(A^\top)^kR_iA^k$. If in addition the policy dependence $\Pi_i\mapsto R_i(\Xi(\Pi_i);\Pi_i)$ is Lipschitz with constant $L$ and $\delta\|A\|^2+\delta L<1$, the coupled map has a unique fixed point $(\Xi^\star,\Pi^\star)$, the stationary Markov BCE, and the iteration $\Pi^{(k+1)}=\delta A^\top\Pi^{(k)}A+R_i(\Xi(\Pi^{(k)});\Pi^{(k)})$ converges to it geometrically.
\end{proposition}
\begin{proof}
For fixed $R_i$, $\|\mathcal{L}(\Pi)-\mathcal{L}(\Pi')\|=\delta\|A^\top(\Pi-\Pi')A\|\le\delta\|A\|^2\|\Pi-\Pi'\|$, a contraction when $\delta\|A\|^2<1$; iterating from $\Pi=0$ gives the Neumann series, convergent under the same bound. Letting $R_i$ vary with $\Pi_i$ through the policy adds a Lipschitz-$L$ term, so the composite map has modulus at most $\delta\|A\|^2+\delta L$; under $\delta\|A\|^2+\delta L<1$ Banach's fixed-point theorem gives existence, uniqueness, and geometric convergence of the stated iteration.
\end{proof}

In the scalar symmetric example of Section~\ref{sec:examples}, \eqref{eq:are} reads $\Pi=\delta A^2\Pi+R(\Pi)$ with contraction modulus $\delta A^2<1$; at a fixed policy $\Pi=R/(1-\delta A^2)$, and the full coupling through $\widetilde\Phi_{ii}=h_0-\delta B^2\Pi$ makes it a scalar algebraic Riccati equation whose stabilising root is the stationary continuation coefficient.

The reduction also settles the computational status of Algorithm~\ref{alg:aps} in this class.

\begin{proposition}[Complexity and convergence of Algorithm~\ref{alg:aps} in the LQG class]\label{prop:complexity}
Consider the LQG class of Theorem~\ref{thm:lqg} with state dimension $n$. (i) In finite horizon, Algorithm~\ref{alg:aps} reduces to $H$ stage problems solved backward: for given $\Pi_{h+1,i}$, the stage-$h$ designer problem is a semidefinite program in the $O(n^2)$ joint second moments of $(\gamma_h,s_h,a_h)$ (the objective is linear and each obedience condition \eqref{eq:lqg_dynobd} affine in these moments, over the positive-semidefinite cone with the marginals of $(\gamma_h,s_h)$ fixed), followed by the Riccati update \eqref{eq:riccati} at cost $O(n^3)$. (ii) In the stationary case, value iteration on \eqref{eq:are} converges geometrically with modulus $q:=\delta\|A\|^2+\delta L<1$ of Proposition~\ref{prop:stationary}, so an $\varepsilon$-accurate fixed point is reached in at most $\lceil\log(\varepsilon_0/\varepsilon)/\log(1/q)\rceil$ iterations, where $\varepsilon_0$ is the initial error.
\end{proposition}
\begin{proof}
(i) For fixed $\Pi_{h+1,i}$ the modified interaction matrix $\widetilde\Phi_h$ is constant, so each obedience condition \eqref{eq:lqg_dynobd} is an affine equality in the entries of the joint covariance of $(\gamma_h,s_h,a_h)$; the designer's stage objective $\E[v_h]$ is linear in the same moments for quadratic $v_h$; and the feasible moments form a spectrahedron. Substituting the optimal stage policy into \eqref{eq:riccati} is one $n\times n$ update. (ii) is the contraction estimate of Proposition~\ref{prop:stationary} applied to the sequence $\Pi^{(k+1)}=\mathcal{R}(\Pi^{(k)})$: $\|\Pi^{(k)}-\Pi^\star\|\le q^{k}\|\Pi^{(0)}-\Pi^\star\|$.
\end{proof}
\begin{center}\rule{0.6\linewidth}{0.4pt}\\[4pt]{\large\bfseries Part II: Markov Information Processes with Learning Agents}\\[2pt]\rule{0.6\linewidth}{0.4pt}\end{center}

\section{Agents Who Observe the State and Learn the Law of Motion}\label{sec:learning}
Part~I assumed the transition $P_h$ is common knowledge. We now let agents be uncertain about it and learn it from experience. Two modelling changes accompany this, and it is worth saying at the outset why they are introduced together. Under known dynamics, observing the persistent state is strategically inert: a far-sighted agent that knows $P_h$ already knows how its action propagates to $s_{h+1}$ and can compute its continuation consequences from its posterior alone, so revealing the realisation of $s$ changes the \emph{form} of the obedience condition (it relocates the term $\cov(a_i,s_i)$ into a known intercept, as we show below) without changing what any agent does. Observability acquires strategic content only when the transition is unknown, since an agent can exploit a superior model of the dynamics only by steering a state it can see. We therefore introduce both changes at once: observability is the assumption that makes learning consequential, and learning is what makes observability matter.

\begin{assumption}[Observed state and known payoffs in Part~II]\label{ass:partII}
Throughout Part~II each agent observes the persistent state $s_h$ at the start of stage $h$, in addition to its recommendation, and knows its own stage payoff $u_{h,i}$. Agents remain uncertain about the transient payoff state $\gamma_h$, about which the recommendation is informative, and about the transition kernel $P_h$.
\end{assumption}

Observing $s_h$ turns the learning target, the conditional law $P_h(\cdot\mid s,\gamma,a)$, into a fully observed estimation problem rather than latent-state inference, so an agent's advantage is cleanly a model of the dynamics rather than an inference about a hidden state. Each agent $i$ holds a model $\hat P^{(i)}_h$ of the transition and best-responds under it. The single change from Part~I is in the continuation differential of the dynamic obedience condition: in \eqref{eq:dynobd}--\eqref{eq:Delta} the term that prices the future,
\[
\begin{aligned}
\widehat\Delta^{\,i}_{h}(s,\gamma,a_i,a_{-i};a_i')\ :=\ \E_{s'\sim \hat P^{(i)}_h(\cdot\mid s,\gamma,a_i,a_{-i})}\!\big[\hat W_{h+1,i}(s')\big]
-\ \E_{s'\sim \hat P^{(i)}_h(\cdot\mid s,\gamma,a_i',a_{-i})}\!\big[\hat W_{h+1,i}(s')\big],
\end{aligned}
\]
is computed under the agent's model $\hat P^{(i)}_h$ and its perceived continuation value $\hat W_{h+1,i}$, rather than under the true $P_h$ and $W_{h+1,i}$. The designer, committing against the true dynamics, builds a policy that is obedient under $\Delta^i_h$; an agent evaluates the same recommendation under $\widehat\Delta^{\,i}_h$. The gap between the two is what a superior model of the dynamics is worth, which we now make precise in general and then in closed form for the LQG class.

\section{The Value of Knowing the Dynamics}\label{sec:rent}
Fix a stage $h$, an observed state $s$, a payoff state $\gamma$, and a recommendation $a_i^\star$ to agent $i$. Under its model $\hat P^{(i)}_h$ the agent evaluates an action $a_i'$ by the perceived stage-plus-continuation payoff
\begin{equation}\label{eq:perceived_obj}
\widehat U^{(i)}_h(a_i')\ :=\ \E\Big[\,u_{h,i}(s,\gamma,a_i',a_{-i})\ +\ \delta\,\E_{s'\sim \hat P^{(i)}_h(\cdot\mid s,\gamma,a_i',a_{-i})}\big[\hat W_{h+1,i}(s')\big]\ \Big|\ a_i^\star\Big],
\end{equation}
the analogue of the deviation value $Q_{h,i}$ of \eqref{eq:Q} computed under the agent's own model. We define the agent's \emph{rent} as the perceived gain from its best deviation from the recommendation,
\begin{equation}\label{eq:rent_general}
\mathrm{Rent}_{h,i}\ :=\ \Big[\ \sup_{a_i'}\,\widehat U^{(i)}_h(a_i')\ -\ \widehat U^{(i)}_h(a_i^\star)\ \Big]^{+}.
\end{equation}
This is a well-defined object for any stage payoff and any transition kernel; it requires no linear-quadratic-Gaussian structure.

\begin{proposition}[Rent: non-negativity and zero at the benchmark]\label{prop:rent_general}
For every model $\hat P^{(i)}_h$ the rent \eqref{eq:rent_general} is non-negative. If the agent's model agrees with the truth on the support reachable from $(s,\gamma,a_{-i})$ and its perceived continuation value agrees there, $\hat P^{(i)}_h=P_h$ and $\hat W_{h+1,i}=W_{h+1,i}$ (in particular at the known-dynamics benchmark of Part~I), then $\mathrm{Rent}_{h,i}=0$. Thus the rent is non-negative in general and vanishes exactly when the agent holds no advantage over the model the designer's commitment is built on.
\end{proposition}
\begin{proof}
Non-negativity follows from the $[\,\cdot\,]^{+}$ and the feasibility of $a_i^\star$ in the supremum. If $\hat P^{(i)}_h=P_h$ and $\hat W_{h+1,i}=W_{h+1,i}$ on the reachable support, then $\widehat U^{(i)}_h$ coincides with the true deviation value $Q_{h,i}(s,\cdot)$ of \eqref{eq:Q}. The designer's policy satisfies true-model dynamic obedience (Definition~\ref{def:dynobd}), i.e.\ $a_i^\star\in\arg\max_{a_i'}Q_{h,i}(s,a_i^\star,a_i')$, so the bracketed difference in \eqref{eq:rent_general} is non-positive and the rent is zero.
\end{proof}

The rent is also controlled, in general, by how far the agent's perceived objective sits from the truth, which yields monotonicity and a quantitative form of the zero-at-benchmark property.

\begin{proposition}[Rent bound and monotonicity]\label{prop:rentbound}
Write $U^{(i)}_h(\cdot)=Q_{h,i}(s,a_i^\star,\cdot)$ for the agent's true deviation value \eqref{eq:Q} and $\widehat U^{(i)}_h$ for its perceived counterpart \eqref{eq:perceived_obj}. Then
\begin{equation}\label{eq:rentbound}
\mathrm{Rent}_{h,i}\ \le\ 2\,\big\|\widehat U^{(i)}_h-U^{(i)}_h\big\|_{\infty},
\end{equation}
so the rent is Lipschitz in the agent's misspecification: it vanishes as the perceived objective converges uniformly to the truth, refining Proposition~\ref{prop:rent_general}. Consequently the worst-case rent over a class of models $\mathcal{G}$, $\overline{\mathrm{Rent}}_{h,i}(\mathcal{G}):=\sup_{\hat P^{(i)}_h\in\mathcal{G}}\mathrm{Rent}_{h,i}$, is non-decreasing in $\mathcal{G}$ and equals zero at $\mathcal{G}=\{P_h\}$.
\end{proposition}
\begin{proof}
True-model obedience gives $\sup_{a_i'}U^{(i)}_h(a_i')=U^{(i)}_h(a_i^\star)$. Writing $\eta:=\|\widehat U^{(i)}_h-U^{(i)}_h\|_\infty$, $\sup_{a_i'}\widehat U^{(i)}_h(a_i')\le\sup_{a_i'}U^{(i)}_h(a_i')+\eta=U^{(i)}_h(a_i^\star)+\eta$ and $\widehat U^{(i)}_h(a_i^\star)\ge U^{(i)}_h(a_i^\star)-\eta$, so $\mathrm{Rent}_{h,i}=\big[\sup_{a_i'}\widehat U^{(i)}_h(a_i')-\widehat U^{(i)}_h(a_i^\star)\big]^+\le 2\eta$, which is \eqref{eq:rentbound}. Monotonicity of $\overline{\mathrm{Rent}}_{h,i}$ is a supremum over a larger set, and $\overline{\mathrm{Rent}}_{h,i}(\{P_h\})=0$ by Proposition~\ref{prop:rent_general}.
\end{proof}

Proposition~\ref{prop:rent_general} states the qualitative content in full generality: knowledge of the dynamics is a source of rents, and the rent is zero exactly at the known-dynamics benchmark. Proposition~\ref{prop:rentbound} adds that its size is controlled by the agent's misspecification and monotone in the class of models the designer must deter, the general form of the robustness margin made precise in the LQG case (Proposition~\ref{prop:robust}). What general payoffs do \emph{not} give is a tractable \emph{expression} for the rent. In the LQG class the perceived objective is concave quadratic, the model gap is finite-dimensional, and the rent collapses to a closed form quadratic in the agent's model error, to which we now turn.

\section{Knowing the Dynamics: the LQG Case}\label{sec:rentlqg}
We now specialise to the LQG class of Section~\ref{sec:lqg}, where agent $i$'s model of the transition is a triple $(\hat A^{(i)}_h,\hat B^{(i)}_h,\hat C^{(i)}_h)$ and the rent of \eqref{eq:rent_general} acquires a closed form.

With $s_h$ observed, the conditioning in the LQG obedience condition changes: $s_i$ and $A_h s_h$ are constants, so the term $\cov(a_i,s_i)$ of Theorem~\ref{thm:lqg} is absorbed into the (known) intercept and the true-model obedience condition becomes, for each $i$,
\begin{equation}\label{eq:partII_true}
\sum_{j\in\Ncal}\widetilde\Phi_{h,ij}\,\cov(a_i,a_j)\ =\ \cov(a_i,\gamma_i)\ +\ \delta\,(B_h^{(i)})^\top\Pi_{h+1,i}\,C_h\,\cov(a_i,\gamma),
\end{equation}
covariances now taken conditional on $s_h$. Equation \eqref{eq:partII_true} is the benchmark the designer commits to. It is not a new condition but the dynamic obedience condition of Theorem~\ref{thm:lqg} evaluated under a finer information set: conditioning \eqref{eq:lqg_dynobd} on an observed $s_h$ makes $s_i$ and $A_h s_h$ measurable, so $\cov(a_i,s_i)\to 0$ and the $A_h\cov(a_i,s)$ term drops, leaving exactly \eqref{eq:partII_true}.

Agent $i$'s model is the triple $(\hat A^{(i)}_h,\hat B^{(i)}_h,\hat C^{(i)}_h)$, which induces, through the recursion of Lemma~\ref{lem:quadratic}, a perceived continuation value $\hat W_{h+1,i}$ with quadratic coefficient $\widehat\Pi_{h+1,i}$ and linear coefficient $\widehat\pi_{h+1,i}$. The agent best-responds under its own model. We collect the discrepancy in the \emph{model gap}
\begin{equation}\label{eq:gap}
G_{h,i}\ :=\ \big(\hat A^{(i)}_h-A_h,\ \hat B^{(i)}_h-B_h,\ \hat C^{(i)}_h-C_h,\ \widehat\Pi_{h+1,i}-\Pi_{h+1,i},\ \widehat\pi_{h+1,i}-\pi_{h+1,i}\big),
\end{equation}
the finite-dimensional realisation of the general gap of Proposition~\ref{prop:rent_general}, zero exactly when the agent's model agrees with the truth.

The designer commits to a policy satisfying the true-model condition \eqref{eq:partII_true}, so under the true model every agent's first-order gain from deviating at the recommended action is zero (Theorem~\ref{thm:lqg}). An agent evaluating the same recommendation under its own model perceives a residual. Writing the continuation gradient at the recommended action under a model as $g_{h,i}=2(B_h^{(i)})^\top\Pi_{h+1,i}\bar s'+(B_h^{(i)})^\top\pi_{h+1,i}$ (true) and $\hat g_{h,i}$ (perceived, hats throughout), the agent's perceived first-order gain at the recommended $a_i^\star$ is
\begin{equation}\label{eq:residual}
r_{h,i}\ :=\ \delta\,\E\big[\hat g_{h,i}-g_{h,i}\,\big|\,a_i^\star,s_h\big],
\end{equation}
which is linear in the model gap $G_{h,i}$ and vanishes when $G_{h,i}=0$.

\begin{proposition}[Rent from superior knowledge of the dynamics]\label{prop:rent}
Let $\widehat{\Phi}^{\mathrm{eff}}_{h,ii}:=\Phi_{ii}-\delta(\hat B^{(i)}_h)^\top\widehat\Pi_{h+1,i}\hat B^{(i)}_h>0$ (perceived own-curvature). Then agent $i$'s perceived-optimal deviation from the recommendation is
\begin{equation}\label{eq:dev}
\hat a_i-a_i^\star\ =\ \frac{r_{h,i}}{2\,\widehat{\Phi}^{\mathrm{eff}}_{h,ii}},
\end{equation}
and the per-stage rent it perceives from this deviation is
\begin{equation}\label{eq:rent}
\mathrm{Rent}_{h,i}\ =\ \frac{r_{h,i}^2}{4\,\widehat{\Phi}^{\mathrm{eff}}_{h,ii}}\ \ge\ 0 .
\end{equation}
The rent is quadratic in the residual $r_{h,i}$, hence quadratic in the model gap $G_{h,i}$, and equals zero if and only if $r_{h,i}=0$; in particular it is zero at the known-dynamics benchmark of Part~I. It is the closed form, in the LQG class, of the general rent \eqref{eq:rent_general}.
\end{proposition}
\begin{proof}
Under Assumption~\ref{ass:partII} the agent's perceived stage-plus-continuation objective is concave quadratic in its own action with curvature $2\widehat{\Phi}^{\mathrm{eff}}_{h,ii}>0$; its derivative at the recommended action $a_i^\star$ equals the true derivative (zero, by \eqref{eq:partII_true}) plus the model-induced residual \eqref{eq:residual}, i.e.\ $r_{h,i}$. For a concave quadratic with curvature $\kappa:=2\widehat{\Phi}^{\mathrm{eff}}_{h,ii}$ and derivative $r_{h,i}$ at $a_i^\star$, the maximiser is $a_i^\star+r_{h,i}/\kappa$, giving \eqref{eq:dev}, and the gain over $a_i^\star$ is $r_{h,i}^2/(2\kappa)=r_{h,i}^2/(4\widehat{\Phi}^{\mathrm{eff}}_{h,ii})$, which is \eqref{eq:rent}. The residual is linear in $G_{h,i}$ by \eqref{eq:residual} and \eqref{eq:gap}, so the rent is quadratic in $G_{h,i}$ and vanishes iff $r_{h,i}=0$.
\end{proof}

Proposition~\ref{prop:rent} isolates the channel: the rent lives entirely in the continuation term that the controlled transition created in Part~I. An agent that knows the dynamics better than the designer's committed policy accounts for converts that knowledge into a profitable, undeterred deviation; an agent whose model matches the truth has no such advantage. The designer can restore deterrence only by over-satisfying obedience.

\begin{proposition}[Robust obedience]\label{prop:robust}
Suppose the designer wishes to deter all agents whose model gap lies in a set $\mathcal{G}$ with $\sup_{G_{h,i}\in\mathcal{G}}|r_{h,i}|\le \bar r_{h,i}(\mathcal{G})$. Then it suffices for the designer to commit to a policy that satisfies the true-model first-order condition with a margin: the recommended action $a_i^\star$ is chosen so that the true continuation-adjusted marginal payoff is non-positive over the deviation directions consistent with $\mathcal{G}$, which tightens \eqref{eq:partII_true} to an inequality with slack of order $\bar r_{h,i}(\mathcal{G})$. The slack is monotone in the size of $\mathcal{G}$ and shrinks to zero as $\mathcal{G}\to\{0\}$, recovering \eqref{eq:partII_true}.
\end{proposition}
\begin{proof}
Fix the stage $h$ and agent $i$ and suppress both indices. Let $m(a_i^\star)$ be the true continuation-adjusted marginal payoff of the agent's stage-plus-continuation objective at the recommended action; the true-model condition \eqref{eq:partII_true} is $m(a_i^\star)=0$. An agent with model gap $G_{h,i}\in\mathcal{G}$ perceives instead the marginal $m(a_i^\star)+r_{h,i}$, with $r_{h,i}$ the residual \eqref{eq:residual}; by Proposition~\ref{prop:rent} its perceived-optimal deviation is $r_{h,i}/(2\widehat\Phi^{\mathrm{eff}}_{h,ii})$, taken in the direction $d=\operatorname{sign}(r_{h,i})$, with gain $r_{h,i}^2/(4\widehat\Phi^{\mathrm{eff}}_{h,ii})$ increasing in $|r_{h,i}|$.

To deter the deviation in a direction $d\in\{+,-\}$ that some model in $\mathcal{G}$ would take, it suffices that the perceived directional marginal be non-positive there for every such model, i.e.\ $d\,\big(m(a_i^\star)+r_{h,i}\big)\le 0$ for all $G_{h,i}\in\mathcal{G}$ with $\operatorname{sign}(r_{h,i})=d$. Since $|r_{h,i}|\le\bar r_{h,i}(\mathcal{G})$ by hypothesis, it is enough that $d\,m(a_i^\star)\le-\bar r_{h,i}(\mathcal{G})$: then $d\,(m(a_i^\star)+r_{h,i})\le -\bar r_{h,i}(\mathcal{G})+|r_{h,i}|\le 0$. This replaces the equality \eqref{eq:partII_true}, $m(a_i^\star)=0$, by a condition in which the true marginal carries a slack of magnitude $\bar r_{h,i}(\mathcal{G})$ on each deterred direction; equivalently, committing with this slack bounds the worst-case undeterred gain over $\mathcal{G}$ by $\bar r_{h,i}(\mathcal{G})^2/(4\widehat\Phi^{\mathrm{eff}}_{h,ii})$ (Proposition~\ref{prop:rent}).

Finally $\bar r_{h,i}(\mathcal{G})=\sup_{G_{h,i}\in\mathcal{G}}|r_{h,i}|$ is non-decreasing in $\mathcal{G}$ (a supremum over a larger set) and equals $0$ when $\mathcal{G}=\{0\}$, since the residual \eqref{eq:residual} vanishes at a correct model. Hence the required slack is monotone in the size of $\mathcal{G}$ and shrinks to zero as $\mathcal{G}\to\{0\}$, recovering \eqref{eq:partII_true}.
\end{proof}

We now show that the $\tilde O(\log T)$ rate is a theorem under Assumption~\ref{ass:learning}, which comprises one regularity condition, (L1), and three conditions, (L2)--(L4), each assuming away one coupling of the fully general problem; we then state the unconditional version as the open conjecture.

\begin{assumption}[Stationary learning environment]\label{ass:learning}
\leavevmode
\begin{enumerate}
\item[(L1)] \emph{(Stationary, contractive primitives.)} The primitives are time-homogeneous and satisfy the contraction condition $\delta\|A\|^2+\delta L<1$ of Proposition~\ref{prop:stationary}, so the Riccati map $\theta:=(A,B,C)\mapsto\Pi(\theta)$ of \eqref{eq:are} is well defined and, by the implicit-function theorem, locally Lipschitz with constant $L_\Pi$ on a neighbourhood of the true $\theta$.
\item[(L2)] \emph{(On-path estimation.)} Agents follow recommendations on path and estimate $\theta$ from the realised obedient transitions $(s_k,\gamma_k,a_k,s_{k+1})_{k\le t}$ by ridge least squares; the rent of Proposition~\ref{prop:rent} is the counterfactual one-shot deviation gain and is not realised, so the estimation data is independent of any agent's deviation.
\item[(L3)] \emph{(Exogenous persistent excitation.)} The designer's stationary policy injects exogenous excitation so that the regressor information matrix $\Lambda_t=\lambda I+\sum_{k\le t}z_kz_k^\top$, $z_k:=(s_k,\gamma_k,a_k)$, satisfies $\lambda_{\min}(\Lambda_t)\ge\lambda_0 t$ for all $t\ge t_0$, some $\lambda_0,t_0>0$.
\item[(L4)] \emph{(Common environment; heterogeneous estimators; sub-Gaussian noise.)} All agents estimate the same transition $\theta=(A,B,C)$ from the common observed obedient transitions and treat it as fixed (they do not form beliefs about one another's estimates), but may use different ridge regularizers $\lambda^{(i)}>0$ and prior means, so the estimators are heterogeneous across agents. The regressors are bounded ($\|z_k\|\le z_{\max}$, from stability and bounded controls), and the transition noise $\varepsilon_k$ is mean-zero, i.i.d., $\sigma_\varepsilon$-sub-Gaussian.
\end{enumerate}
\end{assumption}

\begin{remark}[Homogeneous learners as a special case]
(L4) nests the homogeneous-learner setting: taking a common regularizer $\lambda^{(i)}\equiv\lambda$ and a common prior mean across agents makes the estimators identical on the common data, the per-agent constants $C_1^{(i)}$ in Theorem~\ref{thm:regret} coincide, and a single representative agent's estimation bound covers all agents. Heterogeneity thus costs only the constant $\max_i C_1^{(i)}$, never the rate.
\end{remark}

\begin{theorem}[Logarithmic rent under exogenous excitation]\label{thm:regret}
Under Assumption~\ref{ass:learning}, there is a constant $c$ (depending on $p,\lambda_0,L_\Pi,\sigma_\varepsilon,z_{\max}$, the rent constants of Proposition~\ref{prop:rent}, and the agents' regularizers and priors) such that, with probability at least $1-\delta'$, for all $t\ge t_0$ \emph{every agent's} per-stage rent satisfies
\begin{equation}\label{eq:rent_rate}
\mathrm{Rent}_{t}\ \le\ c\,\frac{\log(t/\delta')}{t},\quad\text{and hence}\quad \sum_{t=t_0}^{T}\mathrm{Rent}_{t}\ \le\ c'\,\log^2(T/\delta')\ =\ \tilde O(\log T);
\end{equation}
summing over the finite set of agents preserves the $\tilde O(\log T)$ rate.
\end{theorem}
\begin{proof}
Write the transition as the linear regression $s_{k+1}=\theta z_k+\varepsilon_k$ with $\theta:=(A,B,C)$ and $p:=\dim z_k$ the regressor dimension, and write $G_t^{(i)}:=\hat\theta_t^{(i)}-\theta$ for the estimated-parameter component of agent $i$'s model gap \eqref{eq:gap}. \\\emph{Step 1 (estimation).} The self-normalised least-squares bound \cite{abbasi2011improved} applies to each agent's ridge estimator separately. Agent $i$'s information matrix $\Lambda_t^{(i)}=\lambda^{(i)}I+\sum_{k\le t}z_kz_k^\top$ shares the common regressors $z_k$, differing across agents only through the regularizer $\lambda^{(i)}$; under (L2),(L4) there is, with probability $\ge1-\delta'$, simultaneously for all $t$, a confidence width $\beta_t^{(i)}=\sigma_\varepsilon\sqrt{2\log(\det(\Lambda_t^{(i)})^{1/2}\det(\lambda^{(i)} I)^{-1/2}/\delta')}+(\lambda^{(i)})^{1/2}\|\theta-\theta_0^{(i)}\|$ (the second term, the prior-induced bias, is a constant) with $\|\hat\theta_t^{(i)}-\theta\|_{\Lambda_t^{(i)}}\le\beta_t^{(i)}$. Bounded regressors give $\det\Lambda_t^{(i)}=O(t^{p})$, so $\beta_t^{(i)}=O(\sqrt{p\log(t/\delta')})$, and by (L3) the common data satisfies $\lambda_{\min}(\Lambda_t^{(i)})\ge\lambda_0 t$, whence
\[
\|G_t^{(i)}\|\ :=\ \|\hat\theta_t^{(i)}-\theta\|_2\ \le\ \frac{\beta_t^{(i)}}{\sqrt{\lambda_{\min}(\Lambda_t^{(i)})}}\ \le\ \frac{\beta_t^{(i)}}{\sqrt{\lambda_0 t}}\ \le\ C_1^{(i)}\sqrt{\frac{p\log(t/\delta')}{\lambda_0\,t}} ,
\]
with $C_1^{(i)}$ depending on $\lambda^{(i)}$ and the prior bias. As there are finitely many agents, $C_1:=\max_i C_1^{(i)}$ is finite and the bound holds uniformly in $i$; a union bound over the finitely many agents replaces $\delta'$ by $\delta'/n$, which only affects constants.
\\\emph{Step 2 (Riccati Lipschitz).} By (L1), for $t$ large enough that $\|G_t^{(i)}\|$ lies in the Lipschitz neighbourhood, $\|\widehat\Pi_t^{(i)}-\Pi\|\le L_\Pi\|G_t^{(i)}\|$. \\\emph{Step 3 (rent).} By Proposition~\ref{prop:rent} agent $i$'s residual $r_t^{(i)}$ is linear in the continuation-gradient gap, $|r_t^{(i)}|\le C_2 L_\Pi\|G_t^{(i)}\|$, and $\mathrm{Rent}_t^{(i)}=(r_t^{(i)})^2/(4\widehat\Phi^{\mathrm{eff}})\le C_3\|G_t^{(i)}\|^2$. Combining with Step 1 and the uniform constant $C_1=\max_i C_1^{(i)}$, every agent's rent obeys $\mathrm{Rent}_t\le C_3C_1^2\,\tfrac{p}{\lambda_0}\,\tfrac{\log(t/\delta')}{t}=:c\,\tfrac{\log(t/\delta')}{t}$.\\ \emph{Step 4 (sum).} $\sum_{t=t_0}^{T}c\,\tfrac{\log(t/\delta')}{t}\le c\!\int_{t_0}^{T}\tfrac{\log(s/\delta')}{s}\,ds=\tfrac{c}{2}\big[\log^2(T/\delta')-\log^2(t_0/\delta')\big]=O(\log^2(T/\delta'))$; since the agent set is finite, summing this bound over agents preserves the $\tilde O(\log T)$ rate.
\end{proof}

Of Assumption~\ref{ass:learning}, (L1) is pure regularity: it makes the stationary Riccati map well defined and locally Lipschitz. Assumptions (L2)--(L4) are exactly the three couplings that a fully general treatment must instead control: (L2) replaces the \emph{self-perturbed} estimation data of a deviating agent with on-path data and a counterfactual rent; (L3) replaces the \emph{endogenous} excitation the designer's own policy generates with an exogenous lower bound; and (L4), while allowing heterogeneous estimators, holds the environment common and exogenously fixed, ruling out the \emph{mutual} learning in which each agent must also track the others' evolving models, the entanglement whose asymmetric, co-evolving gaps make obedience bind unevenly and turn the design into a moving target. Removing (L2)--(L4) yields the open problem.

\begin{conjecture}[Unconditional logarithmic rent]\label{conj:regret}
The cumulative-rent bound of Theorem~\ref{thm:regret}, $\sum_{t\le T}\mathrm{Rent}_t=\tilde O(\log T)$, continues to hold when (L2)--(L4) are dropped: when the excitation is endogenous to the designer's policy, agents learn from their own (possibly off-path) data, and agents' learning is mutually entangled (each tracking the others' evolving models), so that the policy, the estimators, and the obedience constraints co-evolve as an informational arms race.
\end{conjecture}

\textbf{A design tension.} Persistent excitation (L3) is double-edged: the recommendation variance that keeps obedience comfortably satisfied is also what feeds the agents' estimators, accelerating their learning and the rent. A designer who wished to slow learning would deliberately under-excite, trading present obedience slack for a slower-growing rent. This is the discrete-time, multi-agent face of the deliberate information-rate throttling that \cite{ctpersuasion2024} leave open in continuous time, and it is the sense in which Part~II's frontier is a design problem and not merely an estimation one.

\section{Two Worked Examples: Evacuation and Power Coordination}\label{sec:examples}
We close with two scalar, two-agent LQG Markov information processes that make Theorem~\ref{thm:lqg} and Proposition~\ref{prop:rent} fully explicit. Each borrows its structure from a setting that can be modelled as a Markov information process (a congestion/evacuation setting and a power-coordination setting, studied as continuous time stochastic Stackelberg control problems by \cite{sezer2026hurricane} and \cite{sezer2026power} respectively), but the examples below are self-contained illustrations of the present theory and use none of those papers' results. We take symmetric agents with own-curvature $\Phi_{ii}=h_0>0$ and cross-coupling $\Phi_{ij}=h_1$ ($i\neq j$), and a diagonal transition \eqref{eq:transition} with scalar gains $A_h=A$, $B_h=B$, $C_h=C$ acting on each agent's own coordinate, so each agent's action moves only its own state. Writing $\Pi$ for the (symmetric, scalar) continuation coefficient $\Pi_{h+1,i}$ of Lemma~\ref{lem:quadratic}, Theorem~\ref{thm:lqg} specialises to the per-agent obedience condition
\begin{equation}\label{eq:ex_obd}
\begin{aligned}
\underbrace{(h_0-\delta B^2\Pi)}_{\widetilde\Phi_{ii}}\,\var(a_i)\ +\ \underbrace{h_1}_{\widetilde\Phi_{ij}}\,\cov(a_i,a_j)\ =\ \cov(a_i,\gamma_i)+\cov(a_i,s_i)
+ \delta B\Pi\big(A\,\cov(a_i,s_i)+C\,\cov(a_i,\gamma_i)\big),
\end{aligned}
\end{equation}
with $\Pi$ solving the backward recursion $\Pi_{h,i}=\delta A^2\Pi_{h+1,i}+R_{h,i}$. The own-curvature is renormalised to $\widetilde\Phi_{ii}=h_0-\delta B^2\Pi$ (the controlled transition makes a unit of own action worth $\delta B^2\Pi$ in continuation), the cross-coupling $h_1$ is untouched because the channels are decoupled, and the bracketed term is the continuation coupling of the recommendation to the states. Setting $\delta=0$ or $B=0$ erases both effects and returns the static condition $h_0\var(a_i)+h_1\cov(a_i,a_j)=\cov(a_i,\gamma_i)+\cov(a_i,s_i)$.

\begin{example}[Evacuation: a two-zone staggering example]\label{ex:evac}
Two evacuation zones $i\in\{1,2\}$ share a downstream corridor. Agent $i$'s action $a_i$ is its egress rate; $\gamma_i$ is the believed local hazard (the storm advisory the designer shapes); $s_i$ is the zone's at-risk backlog. The shared corridor makes simultaneous egress costly, so the cross-coupling is positive, $h_1>0$ (a congestion complementarity penalising co-movement of $a_1,a_2$). Faster egress drains the backlog, so the own channel has $B<0$ (action reduces next-period $s_i$), and a backlog persists, $A\in(0,1)$; the advisory feeds the backlog with $C>0$. Because $B<0$ and (for an exposure-minimising continuation) $\Pi>0$, the renormalised own-curvature $\widetilde\Phi_{ii}=h_0-\delta B^2\Pi<h_0$: a far-sighted zone is \emph{more} willing to evacuate than its stage payoff alone implies, since draining the backlog today lowers its continuation exposure. The designer's obedient recommendation must account for this: with the positive $h_1$, the second moments solving \eqref{eq:ex_obd} push $\cov(a_1,a_2)$ down (anti-synchronised release), the one-period analogue of a staggered evacuation order. In Part~II a zone that has learned the corridor dynamics (a sharper $\widehat\Pi$ than the designer's commitment) perceives a residual $r_{h,i}=2\delta B(\widehat\Pi-\Pi)\,\bar s_i$ and earns
\begin{equation}\label{eq:ex_rent}
\mathrm{Rent}_{h,i}=\frac{\delta^2 B^2(\widehat\Pi-\Pi)^2\,\bar s_i^{\,2}}{4\,(h_0-\delta B^2\widehat\Pi)},
\end{equation}
i.e.\ a zone that better understands how its egress clears the corridor can profitably jump its slot, exactly the synchronisation the staggered order is meant to prevent; the rent grows with the backlog $\bar s_i$ it is sitting on and vanishes when its model matches the designer's.
\end{example}

\begin{example}[Power coordination: a two-area reserve example]\label{ex:power}
Two interconnected areas $i\in\{1,2\}$ are tied by an intertie. Agent $i$'s action $a_i$ is its reserve provision / inter-area export; $\gamma_i$ is the believed local stress (the weather advisory the operator shapes); $s_i$ is the area's reserve-shortfall state. Mutual aid makes the actions substitutes across areas: when one area exports, it relieves the other, so the cross-coupling is negative, $h_1<0$ (an effort the partner can free-ride). Providing reserve improves the own shortfall state, $B>0$, the shortfall is persistent, $A\in(0,1)$, and stress worsens it, $C>0$. Here $\widetilde\Phi_{ii}=h_0-\delta B^2\Pi<h_0$ again, but the economics differ: a far-sighted area values reserve more because today's provision builds a continuation buffer. Under common (systemic) stress the advisory $\gamma_1,\gamma_2$ is correlated, so $\cov(a_i,\gamma_i)$ on the right of \eqref{eq:ex_obd} is large and the continuation term $\delta B\Pi C\,\cov(a_i,\gamma_i)$ \emph{amplifies} it, echoing the way disclosure and coupling can reinforce one another under common risk. In Part~II an area whose model of the tie / reserve dynamics outruns the operator's commitment earns the same rent \eqref{eq:ex_rent}; with $h_1<0$ the binding incentive is reserve \emph{withholding}: an area that has learned it can lean on its neighbour's continuation buffer perceives $r_{h,i}=2\delta B(\widehat\Pi-\Pi)\bar s_i$ and under-provides, a form of strategic withholding that, absent a transfer, the designer here can deter only by the obedience slack of Proposition~\ref{prop:robust}.
\end{example}

The two examples share one obedience condition \eqref{eq:ex_obd} and one rent formula \eqref{eq:ex_rent}; only the signs of the primitives differ, $h_1>0$ (congestion complementarity) for evacuation versus $h_1<0$ (mutual-aid substitutability) for power, and that single sign flip turns the binding Part~II deviation from premature synchronisation into reserve withholding.

\subsection{Numerical illustration: utilitarian and maximin design}\label{sec:numerics}

We instantiate the congestion example ($h_1>0$) and run Algorithm~\ref{alg:aps} in its stationary LQG form (Proposition~\ref{prop:complexity}). Primitives: $h_0=1$, $h_1=0.4$, $A=0.4$, $B=0.25$, $C=0.3$, $\delta=0.9$, transition-noise variance $\sigma_\varepsilon^2=0.05$, $d_i\equiv 0$, and \emph{asymmetric} payoff-state variances $\sigma_{\gamma_1}^2=1.5$, $\sigma_{\gamma_2}^2=0.5$, so the two designer objectives we compare, utilitarian welfare $\sum_i\E[u_i]$ and worst-case (maximin) welfare $\min_i\E[u_i]$, pull in different directions. Recommendations are drawn from the linear-Gaussian class $a_i=\alpha_i\gamma_i+\kappa_i s_i+\eta_i$ with correlated randomisation $(\eta_1,\eta_2)$, and, consistent with the compactness maintained in Proposition~\ref{prop:optexist}, we impose the action bound $\E[a_i^2]\le 4$ on every design (without it the unconstrained first best is unbounded: the designer would exploit unboundedly large anti-correlated randomisation near the stability boundary). Table~\ref{tab:numerics} reports stationary per-stage welfare; obedience residuals at all reported optima are below $10^{-12}$.

\begin{table}[t!]
\caption{Stationary per-stage welfare across designs in the congestion instance. The no-disclosure and first-best rows are the benchmarks of Proposition~\ref{prop:valuebounds}.}
\centering
{\begin{tabular}{lcccc}\toprule
Design & $\E[u_1]$ & $\E[u_2]$ & Total & Worst agent\\ \midrule
No disclosure ($V^{\mathrm{bab}}$) & $0$ & $0$ & $0$ & $0$\\
Maximin optimum & $1.652$ & $1.198$ & $2.850$ & $1.198$\\
Utilitarian optimum ($J^\star$) & $2.386$ & $1.042$ & $3.428$ & $1.042$\\
First best ($V^{\mathrm{fb}}$) & $2.657$ & $0.922$ & $3.579$ & $0.922$\\ \bottomrule
\end{tabular}}
\label{tab:numerics}
\end{table}

Three observations. First, the computed values realise the sandwich of Proposition~\ref{prop:valuebounds}, $V^{\mathrm{bab}}=0\le J^\star=3.428\le V^{\mathrm{fb}}=3.579$: in this instance the value of information design ($3.428$) is large while the price of obedience ($0.151$) is small; the designer loses little by having to respect incentives, but everything by staying silent. Second, the maximin design purchases $+0.157$ for the worst agent at a cost of $0.578$ in total welfare, and it does so by reallocating informativeness toward the low-variance agent; notably, the \emph{first best} treats the worst agent worse than either obedient design ($0.922<1.042$), so in this instance the obedience constraint itself partially protects the weaker agent. Third, Figure~\ref{fig:numerics} (left) shows the value iteration of Proposition~\ref{prop:complexity}(ii) on the instance: the error decays geometrically with empirical modulus $\hat q\approx 0.61$, reaching $10^{-9}$ in $40$ iterations. The right panel simulates a learning agent under the utilitarian policy: the agent ridge-estimates $\theta=(A,B,C)$ from its own trajectory (regularizer $\lambda=1$), and the cumulative squared estimation error $\sum_{k\le t}\|\hat\theta_k-\theta\|^2$ (the driver of the rent bound in Theorem~\ref{thm:regret} through $\mathrm{Rent}_t\le C\|G_t\|^2$) tracks the $c\log^2 t$ envelope of the theorem.

\begin{figure}[t!]
\centering
\includegraphics[width=\textwidth]{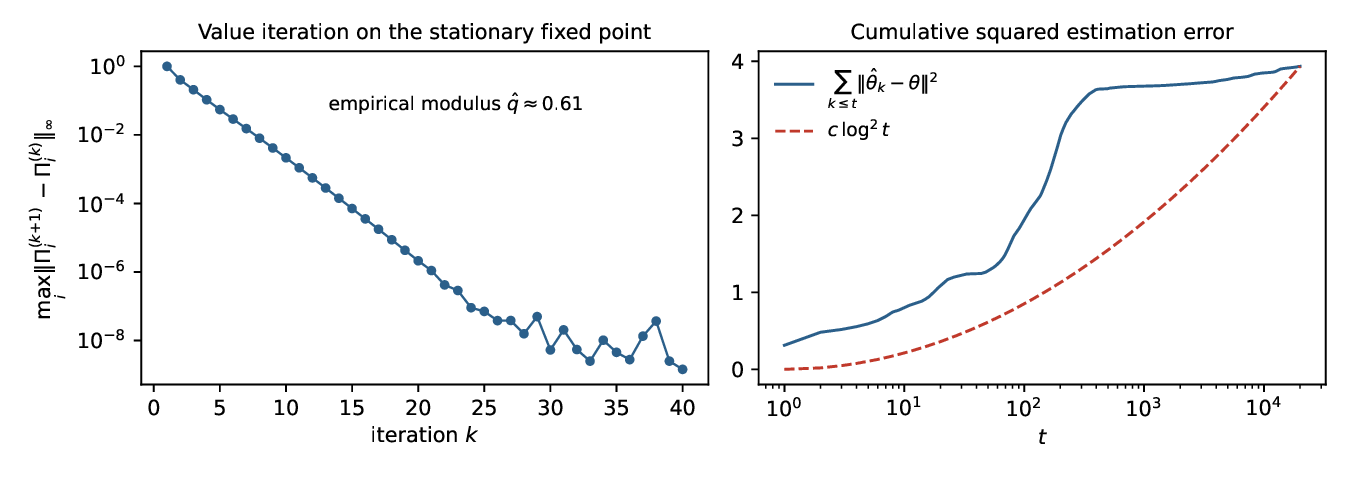}
\caption{Left: geometric convergence of value iteration on the stationary fixed point \eqref{eq:are} in the congestion instance (Proposition~\ref{prop:complexity} (ii)). Right: cumulative squared estimation error of a learning agent under the utilitarian design against the $c\log^{2}t$ envelope of Theorem~\ref{thm:regret} ($c$ fitted at the endpoint, $T=2\times 10^4$).}
\label{fig:numerics}
\end{figure}

\section{Conclusion}
We lifted the Bergemann--Morris obedience condition to a controlled-Markov environment with far-sighted agents. Part~I established the known-dynamics benchmark: a dynamic obedience condition adding a continuation-value differential to the static one, a dynamic revelation principle justifying the recommendation-policy formulation, and a recursive designer problem with the agents' promised utilities as state variables, solved by a set-valued backward-induction algorithm (Algorithm~\ref{alg:aps}) whose optimum is attained and lies between the no-disclosure and first-best values. In the LQG class obedience is a covariance condition with a modified interaction matrix, continuation values are quadratic, and the stationary case is a discounted algebraic Riccati equation, all collapsing to the static case when actions cannot move the state.

Part~II showed that this benchmark is the reference point for an agent that learns the dynamics: knowledge of the transition is a source of rents (non-negative, bounded by the agent's misspecification, and zero exactly at the benchmark) that in the LQG class are quadratic in the model error and deterred only by satisfying obedience with slack. Under stationary primitives and persistent excitation the cumulative rent obeys a logarithmic bound (Theorem~\ref{thm:regret}); the unconditional bound, under endogenous excitation, self-perturbed data, and mutually entangled learning, is the open problem we leave as Conjecture~\ref{conj:regret}. 

\bibliographystyle{unsrtnat}
\bibliography{ref}

\end{document}